\documentclass[]{amsart}

\usepackage{amssymb,amsfonts}
\usepackage[all,arc]{xy}
\usepackage{enumerate}
\usepackage{mathrsfs}
\usepackage{hyperref}
\hypersetup{
    colorlinks,
    citecolor=black,
    filecolor=black,
    linkcolor=red,
    urlcolor=black
}
\usepackage{cleveref}
\usepackage{tikz-cd}
\usepackage{bbold}
\usepackage{marvosym}
\usepackage{color}   
\usepackage{graphicx}

\newtheorem{Theorem}{Theorem}[section]
\newtheorem{Corollary}[Theorem]{Corollary}

\newtheorem{Lemma}[Theorem]{Lemma}

\theoremstyle{definition}
\newtheorem{Definition}[Theorem]{Definition}

\newtheorem{Notation}[Theorem]{Notation}

\theoremstyle{remark}
\newtheorem{Remark}[Theorem]{Remark}

\DeclareMathOperator{\Spec}{Spec}
\DeclareMathOperator{\Gr}{Gr} 
\DeclareMathOperator{\Fl}{Fl}
\DeclareMathOperator{\SB}{SB}

\DeclareMathOperator{\Sch}{Sch} 
\DeclareMathOperator{\Sets}{Sets}
 
\DeclareMathOperator{\Weyl}{\mathcal{K}}
\DeclareMathOperator{\Schur}{\mathcal{L}}
\DeclareMathOperator{\Tautological}{S}
\DeclareMathOperator{\Quotient}{Q}
\DeclareMathOperator{\Sym}{Sym}
\DeclareMathOperator{\Divided}{Div}
\DeclareMathOperator{\E}{E}
\DeclareMathOperator{\F}{F}
\DeclareMathOperator{\Line}{L}
\DeclareMathOperator{\VectorV}{V}
\DeclareMathOperator{\VectorW}{W}
\DeclareMathOperator{\Affine}{\mathbb{A}}
\DeclareMathOperator{\Integers}{\mathbb{Z}}
\DeclareMathOperator{\Identity}{Id}
\DeclareMathOperator{\Right}{\mathsf{R}}
\DeclareMathOperator{\Left}{\mathsf{L}}
\DeclareMathOperator{\Dqc}{D_{qc}}

\DeclareMathOperator{\RHom}{\mathsf{R}Hom}
\DeclareMathOperator{\RHomi}{\mathsf{R}^iHom}
\DeclareMathOperator{\RsheafHom}{\mathsf{R}\mathcal{H}om}

\DeclareMathOperator{\IM}{Im}
\DeclareMathOperator{\D}{D}
\DeclareMathOperator{\M}{M}
\DeclareMathOperator{\Cohomology}{H}

\DeclareMathOperator{\Br}{Br}

\newcommand{\Gm}{{\mathbb G}_m}

\makeatletter
\let\c@equation\c@Theorem
\makeatother
\numberwithin{equation}{section}

\title[Generalized Severi-Brauer]{Semiorthogonal decompositions for Generalized Severi-Brauer Schemes}

\author{Ajneet Dhillon}
\email{adhill3@uwo.ca}

\author{Sayantan Roy-Chowdhury}
\email{sroycho4@uwo.ca}

\begin{document}
\maketitle

\begin{abstract}
    The purpose of this paper is to use conservative descent to study semi-orthogonal decompositions 
    for some homogeneous varieties over general bases.
    We produce a semi-orthogonal decomposition for the bounded derived category of coherent sheaves 
    on a generalized Severi-Brauer scheme. This extends known results for Sever-Brauer varieties and 
    Grassmanianns. We use our results to construct semi-orthogonal decompositions for flag varieties over arbitrary 
    bases. This generalises a result of Kapranov. 
\end{abstract}

\tableofcontents

\section{Introduction}
In this article,
we study semi-orthogonal decompositions for Grassmannians and flag varieties over general bases, see \ref{Theorem Derived Category of Grassmannian over algebraic stack for Dqc}, 
\ref{Theorem Decomposition of Derived Category of Generalised Severi Brauer Varieties} and \ref{Corollary Semiorthogonal Decomposition of Flags}. These 
notions are recalled in section \ref{s:dc}.

These problems have a long history. The first result in this area was Beilinson's theorem, \cite{beilinson}, which showed that ${\mathbb P}^n$ has a full exceptional collection.
Later Kapranov, see \cite{Kapranov1}, proved a similar result for a Grassmannian over a field of characteristic zero. Kapranov's calculation makes use of the Borel-Bott-Weil
theorem which is necessarily a characteristic zero result. For Grassmanianns over an arbitrary fields, Buchwitz, Leuschke and Van den Bergh, 
see \cite{Grassmannian-arbitrary-characteristic},
have shown that such a Grassmannian has a full exceptional collection. This result was further strengthed by Efimov, see \cite{efimov}, who showed that such a result 
holds over the integers. The problem has been further studied in the relative setting. For projective and Grassman bundles in characteristic zero an exceptional 
collection is produced in \cite{orlov}.

Recently, arithmetic consequences of decompositions of the derived category have been observed, see \cite{AB} and \cite{bdm}. This leads naturally to the question of 
constructing decompositions for twisted forms of varieties. A Severi-Brauer scheme is a twisted form of a projective bundles. 
Semi-orthogonal decompositions for  Severi-Brauer 
schemes are constructed  in \cite{bernardara}. In characteristic zero, refinements, such as tilting sheaves are discussed in \cite{blunk}. 
Generalised Severi-Brauer schemes are twisted forms of Grassmannians. Over fields of characteristic zero, semi-orthogonal decompositions
for generalised Severi-Brauer varieties are studied in \cite{baek}. 

In this paper, we will add to this body of work. Our main tool is the notion of conservative descent introduced by 
Bergh and Schn\"{u}rer, see \cite{Bergh-Schnurer-Conservative-Descent}. This allows one to descend 
decompositions along conservative covers. They have applied their theory to gerbes and Severi-Brauer schemes (stacks), see \cite{Bergh-Schurer-Decomposition-Of-Gerbes},
to produce semi-orthogonal decompositions. We will review their results in \ref{ss:conservative}. The methods and ideas of Bergh and Schn\"{u}rer are 
essential for this paper.

In section \ref{s:grassmanian} we apply the machinery of 
\cite{Bergh-Schnurer-Conservative-Descent} and \cite{Grassmannian-arbitrary-characteristic}
to a Grassmannian, potentially over the integers, to recover one of the results in \cite{efimov} via different means.
In the final section we are able to extend this to a flag variety over the integers, extending the result in \cite{Kapranov2}. From our understanding, 
this result is new and other authors have not considered flags over the integers.
In \ref{s:gsb},  we are able to produce decompositions for  generalised Severi-Brauer stacks which is a vast extension of the main results of \cite{baek}.
The last section of the paper, explains why semi-orthogonal decompositions, rather than exceptional collections, are the best available decompositions for 
twisted forms, see \ref{Corollary Generalised Severi Brauer Over a Field}.

\section{Notation}

Let $k$ and $n$ be integers satisfying $0$ $<$ $k$ $<$ $n$ and let $[n]$ denote the set ${1,2,\hdots,n}$. We denote by $\Sch$ the category of schemes
and by $\Sets$ the catgory of sets. We will usually reserve the letters $K$ for a field, $X$, $Y$ for stacks, $\Spec(R)$ for affine
schemes, $\E$ and $\F$ for sheaves, $\VectorV$ and $\VectorW$ for vector bundle, $\Line$ for line bundle, $\Tautological$ and $\Quotient$ for 
tautological and the quotient bundle respectively. The Schur and Weyl functors are denoted by $\Schur^{\alpha}$ and $\Weyl^{\alpha}$ respectively.

\section{The Grassmannian}
\label{Section Grassmannian}

\begin{Definition}
    \label{Definition Grassmannian Functor}
    Given a vector bundle $\VectorV$ over an algebraic stack $X$, we define a functor from the category of schemes over 
    $X$ to the category of sets
    \[     Gr(k,\VectorV)  : \Sch/X  \rightarrow \Sets          \]
    that associates to a scheme $T$ over $X$ 
    \[ p : T \rightarrow X \]
    the isomorphism classes of surjections
    \[    f : p^{*}(\VectorV) \twoheadrightarrow \Quotient \]
    where  $\Quotient$ is a finite locally free $\mathcal{O}_T$-module of rank $n - k$. In the case $\VectorV$ is free of rank $n$,
    we denote this by $Gr(k,n)$.
\end{Definition}

\begin{Theorem}
\label{Representability of the Grassmannian functor}
 Let $P_{n-k}$ be the set of subsets of $[n]$ of cardinality $n - k$. Then the functor $Gr(k,n)$ defined in
\ref{Definition Grassmannian Functor} has an open cover by representable open subfunctors $\{F_I\}_{I \in P_{n-k}}$. Thus,
it is representable by a scheme.
\end{Theorem}

\begin{proof}
    \cite[\href{https://stacks.math.columbia.edu/tag/089T}{Lemma 089T}]{stacks-project}.
\end{proof}

The scheme $\Gr(k,n)$ representing the functor $Gr(k,n)$ over $\Spec(\Integers)$ is called the Grassmannian over the integers.
We denote its base change to an affine scheme $\Spec(R)$ (resp. algebraic stack $X$) as $\Gr(k,n)_{R}$ (resp. $\Gr(k.n)$). 
The universal quotient bundle $\Quotient$ over $\Gr(k,n)_{X}$ fits into the tautological exact sequence
\[ 0 \rightarrow \Tautological \rightarrow \mathcal{O}_{\Gr(k,n)_{X}}^{\oplus n} \rightarrow \Quotient \rightarrow 0. \]
More generally, the functor $Gr(k,\VectorV)$ is also representable and we denote the corresponding algebraic stack
over $X$ as $\Gr(k,\VectorV)$, it also has a universal quotient bundle $\Quotient$ over it that fits into an exact sequence 
\[ 0 \rightarrow \Tautological \rightarrow \mathcal{O}_{\Gr(k,\VectorV)_{X}}^{\oplus n} \rightarrow \Quotient \rightarrow 0. \]

\begin{Remark}
\label{Remark Description of Open Sets of Grassmannian}
The proof in \cite[\href{https://stacks.math.columbia.edu/tag/089T}{Lemma 089T}]{stacks-project}  reveals that
the open subfunctors $F_I$ can be represented by open subsets $U_I$ of $\Gr(k,n)$ isomorphic to the affine space 
$\Affine_{\Integers}^{k(n-k)}$. We now give a concrete description of the open sets after fixing a basis
 $\{e_1,e_2,\hdots,e_n\}$ of $\mathcal{O}_{\Gr(k,n)}^{\oplus n}$.  The open sets $U_I$ can then be described as subsets of the
Grassmannian $\Gr(k,n)$ where the restriction of the quotient map $f$ : $\mathcal{O}_{\Gr(k,n)}^{\oplus n}$ $\rightarrow$
$\Quotient$ to the subbundle generated by $\{e_{i}\}_{i \in I}$ is surjective. Furthermore, suppose 
\[ I = \{i_1, i_2, \hdots i_{n-k}\}   \hspace{5em}       [n]\setminus I = \{j_1, j_2, \hdots j_k\}    \]
such that both the sequences are increasing. Then given affine coordinates 
$\{x_{lm}\}$ for $l \in [k]$ and $m \in [n-k]$ on a point $x$ in $U_I$, the set
\[            \{e_{j_l} + \sum_{m \in [n-k]} x_{lm} e_{i_m}\}_{l \in [k]}                              \]
forms a basis for the tautological bundle $\Tautological$. The dual of the Quotient bundle $\Quotient^{\vee}$ on the other hand
has a basis given by 
\[             \{  e_{i_m}^* - \sum_{l \in [k]} x_{lm} e^*_{j_l}\}_{m \in [n-k]}    \]
where $\{e^*_i\}_{i \in [n]}$ is the basis of  $(\mathcal{O}_{\Gr(k,n)}^{\oplus n})^{\vee}$ satisfying
\[ e_i^*(e_j) = \delta_{ij}.  \]
We can base change and assume the same description for corresponding open sets $U_i$ of $\Gr(k,n)_R$. 
\end{Remark}

\section{Derived Categories}\label{s:dc}
\subsection{Triangulated Category}
We fix $K$ to be a field. What follows are the definitions of exceptional sequence, admissible category
and semiorthogonal decomposition of a triangulated category $\mathcal{T}$.

\begin{Definition}
\label{Definition Classical Generation of Triangulated Catgory}
We say a family of objects $E_{1}, E_{2} \hdots E_{n}$ in a triangulated category $\mathcal{T}$ classically generates
$\mathcal{T}$  if the smallest strictly full, thick triangulated subategory containing these objects is $\mathcal{T}$ itself.
\end{Definition} 

\begin{Definition}
\label{Definition Exceptional Sequence}
		A sequence of objects $E_{1}, E_{2} \hdots E_{n}$ of a $K$-linear triangulated category $\mathcal{T}$
        is defined to be exceptional if 
		\begin{enumerate}
			\item They classically generate $\mathcal{T}$
		    \item \[\RHom (E_{i},E_{j}) = 0       \]
	for j $<$ i
	        \item  \[\RHom (E_{i},E_{i}) = K .      \]
        \end{enumerate}
\end{Definition}

\begin{Definition}
\label{Definition Admissible Category}
A full triangulated subcategory $\mathcal{A}$ of $\mathcal{T}$ is defined to be right(left) admissible if the inclusion functor
from $\mathcal{A}$ to $\mathcal{T}$ admits a right(left) adjoint.  
\end{Definition}

\begin{Definition}
\label{Definition Semiorthogonal Decomposition}
    A semiorthogonal decomposition of $\mathcal{T}$ is a sequence of strictly full triangulated subcategories
    $\mathcal{T}_1$, $\mathcal{T}_2$, $\mathcal{T}_3$, $\hdots$, $\mathcal{T}_n$ such that
    \begin{enumerate}
        \item For all $1 \leq i < j \leq n$ and objects $T_i$, $T_j$ in $\mathcal{T}_i$ and $\mathcal{T}_j$ respectively,
        $\RHom (T_{i},T_{j}) = 0$. (This sequence is known as a semiorthogonal sequence)
        \item The smallest strictly full triangulated subcategory of $\mathcal{T}$ containing  $\{\mathcal{T}_1$, $\mathcal{T}_2$,
        $\mathcal{T}_3$, $\hdots$, $\mathcal{T}_n\}$ is $\mathcal{T}$.
    \end{enumerate}
    We write 
    \[ \mathcal{T} = \langle \mathcal{T}_1, \mathcal{T}_2, \mathcal{T}_3, \hdots, \mathcal{T}_n \rangle \]
    in this case.
\end{Definition}

\subsection{Derived Category of Sheaves and Conservative Descent} \label{ss:conservative}

We will first introduce several notions of derived categories over an algebraic stack.

\begin{Notation}
\label{Notation Derived Categories of sheaves}
    Given an algebraic stack $X$, we will work with the following derived categories
     \begin{enumerate}
        \item $\D(X)$ denotes the unbounded derived category of $\mathcal{O}_X$-modules over the $lisse$-$etale$ site.
        \item  Let $\Dqc(X)$ denotes the full subcategory of $\D(X)$
    with quasi-coherent cohomology sheaves.
        \item $\D_{p}(X)$ denotes the full subcategory of perfect complexes.
        \item $\D_{pc}(X)$ denotes the full subcategory of pseudo-coherent complexes.
        \item $\D^{lb}_{pc}(X)$ is the full subcategory of pseudo-coherent complexes with locally bounded cohomology.
        \item $\D_{sg}(X)$ denotes the singularity category which is the Verdier quotient $\D^{lb}_{pc}(X)/\D_{p}(X)$.
    \end{enumerate}
\end{Notation}

We now adapt the definition of Fourier-Mukai Transforms 
\cite[Definition 3.3]{Bergh-Schnurer-Conservative-Descent}  to our case where all the morphisms are quasi-compact, quasi-separated
and representable. First, we
fix some notation. Given a quasi-compact and quasi-separated representable morphism $f$: $X$ $\rightarrow$ $Y$ of  algebraic stacks,
we denote by 
\[      \Right\!f_* : \Dqc(X) \rightarrow \Dqc(Y)                                                            \]
\[      \Left\!f^* : \Dqc(Y) \rightarrow \Dqc(X)                                                            \]                                
the derived functor corresponding to pushforward and pullback respectively. The functor $\Right\!f_*$ has a right adjoint 
which we denote by 
\[      f^{\times} : \Dqc(Y) \rightarrow \Dqc(X). \]

\begin{Definition}
\label{Definition Fourier Mukai Transform}
    A \emph{finite Fourier-Mukai transform of algebraic stacks}  a triple $(p,q,\E)$ where 
    $p:G\to Z$ and $q:G\to Y$ are  quasi-compact and quasi-separated representable morphisms of algebraic stacks
    and $\E$ is an object in $\Dqc(G)$
    such that 
    \begin{enumerate}
        \item $p$ and $q$ are proper and perfect.
        \item $q$ is finite and $q^{\times}(\mathcal{O}_Y)$ is perfect. 
    \end{enumerate}
\end{Definition}

\begin{Remark}
    This definition is from \cite[3.3]{Bergh-Schnurer-Conservative-Descent}. In this paper there are further 
    conditions imposed on $p$ and $q$. However the morphisms $p$ and $q$ in \cite{Bergh-Schnurer-Conservative-Descent} are more 
    general than those considered in this paper. The extra conditions are automatically satisfied in the present context, 
    see \cite[3.7]{Bergh-Schnurer-Conservative-Descent}.
\end{Remark}

A finite Fourier-Mukai transform $(p,q,\E)$ induces a functor
\[ \Phi : \Dqc(Z) \rightarrow \Dqc(Y)   \]
\[ \Phi(\F) = \Right\!q_*(\E \otimes \Left\!p^*(\F))             \]
which we call the Fourier-Mukai functor. It
admits a right adjoint \cite[Section 3, Eqn 3.10]{Bergh-Schnurer-Conservative-Descent}
\[ \Phi^* : \Dqc(Y) \rightarrow \Dqc(Z)   \]
\[ \Phi^*(\F) = \Right\!p_*(\E^{\vee} \otimes q^{\times}(\F))             \]
where $\E^{\vee}$ = $\RsheafHom(\E,\mathcal{O}_G)$ denotes the dual of $\E$. This proves the following lemma

\begin{Lemma}
\label{Lemma Essential Images of Fourier Mukai Functors are admissible}
Essential images of fully faithful Fourier-Mukai functors are right admissible.
\end{Lemma}

We now quote \cite{Bergh-Schnurer-Conservative-Descent} for the two most important properties of Fourier-Mukai functors.
Firstly, semiorthogonal decompositions of derived categories of an algebraic stack $X$ into essential images of Fourier-Mukai 
functors can be checked locally by the geometric conservative descent theorem. Secondly, such decompositions for 
$\Dqc(X)$ implies semiorthogonal decompositions of other derived categories.

\begin{Theorem}
\label{Lemma Geometric Conservative Descent Theorem}
Let $\{Y_i\}_{1 \leq i \leq n}$ and $X$ be algebraic stacks over a base algebraic stack $C$ with Fourier-Mukai functors 
\[      \Phi_i  : \Dqc(Y_i) \rightarrow \Dqc(X).                            \]
Let $u : C' \rightarrow C$ be a faithfully flat morphism with $\{Y'_i\}_{1 \leq i \leq n}$ and $X'$ being the base change of 
$\{Y_i\}_{1 \leq i \leq n}$ and $X$ respectively. Then the base changed Fourier-Mukai functors
\[ \Phi'_i  : \Dqc(Y'_i) \rightarrow \Dqc(X') \]
being fully faithful implies the same for $\Phi_i$. If the essential images of $\Phi'_i$ fit into an 
semiorthogonal sequence implies the same about the essential images of $\Phi_i$. Further, if this semiorthgonal sequence 
corresponding to $\Phi'_i$ gives a semiorthogonal decomposition of $\Dqc(X')$, then the semiorthgonal sequence 
corresponding to $\Phi_i$ gives a semiorthogonal decomposition of $\Dqc(X)$.
\end{Theorem} 
\begin{proof}
    \cite[Theorem 6.1]{Bergh-Schnurer-Conservative-Descent}.  
\end{proof}

\begin{Lemma}
\label{Lemma semiorthogonal decomposition of Dqc gives semiorthogonal decomposition of other derived categories}
Let $\{Y_i\}_{1 \leq i \leq n}$ and $X$ be algebraic stacks over a base algebraic stack $C$
with fully faithful Fourier Mukai functors 
\[      \Phi_i  : \Dqc(Y_i) \rightarrow \Dqc(X)                            \]
such that there is a semiorthogonal decomposition of $\Dqc(X)$ into right admissible subcategories given by the essential images
$\IM(\Phi_i)$ of $\Phi_i$
\[      \Dqc(X) = \langle \IM(\Phi_1), \IM(\Phi_2), \hdots \IM(\Phi_n) \rangle                              ,\] 
then there is a semiorthogonal decomposition of the derived categories $\D_p(X)$, $\D^{lb}_{pc}(X)$ and $\D_{sg}(X)$ into right
admissible subcategories given by essential images of the induced fully faithful functors
\[   \Phi_i^p : \D_p(Y_i) \rightarrow \D_p(X)  \] 
\[   \Phi_i^{pc} :  \D^{lb}_{pc}(Y_i) \rightarrow \D^{lb}_{pc}(X) \]
\[   \Phi_i^{sg} :  \D_{sg}(Y_i) \rightarrow \D_{sg}(X)                                                              \]                             
respectively.
\end{Lemma}
\begin{proof}
\cite[Theorem 6.2]{Bergh-Schnurer-Conservative-Descent}.
\end{proof}

We end this section with two lemmas on generation. 

\begin{Lemma}
    \label{Lemma Generator contained in Semiorthogonal Sequence implies Semiorthogonal Decomposition}
    Let $\mathcal{T}$ be a triangulated category such that an object $\E$ in $\mathcal{T}$ classically generates $\mathcal{T}$. Let
    $\mathcal{T}_1$, $\mathcal{T}_2$, $\mathcal{T}_3$, $\hdots$, $\mathcal{T}_n$ be a semiorthogonal sequence of right admissible 
    subcategories of $\mathcal{T}$. Then there is a semiorthogonal decomposition
    \[ \mathcal{T} = \langle \mathcal{T}_1, \mathcal{T}_2, \mathcal{T}_3, \hdots, \mathcal{T}_n \rangle \]
    if and only if $\E\in \langle \mathcal{T}_1, \mathcal{T}_2, \mathcal{T}_3, \hdots, \mathcal{T}_n \rangle$.
\end{Lemma}
\begin{proof}
    \cite[Lemma 6.4]{Bergh-Schnurer-Conservative-Descent}.
\end{proof}

\begin{Lemma}
    \label{Lemma Ample Line Bundle Generator}
    If a scheme $X$ has an ample line bundle $\Line$, then there exists a vector bundle $\VectorV$ over $X$ that classically generates
     $\Dqc(X)$.
\end{Lemma}
\begin{proof}
    \cite[Lemma 6.5]{Bergh-Schnurer-Conservative-Descent}.
\end{proof}

\section{Some Representation Theory}

We recall some relevant facts from 
\cite{Akin-Buchsbaum-Weyman} and \cite{Grassmannian-arbitrary-characteristic}. 

\begin{Definition}
\label{Definition of partitions}
A partition $\alpha$ = $\{ \alpha_{1}, \alpha_{2} \hdots \alpha_{n}\}$ is a non-increasing sequence of non-negative integers. The 
degree of a partition is defined to be the sum $\sum_{i=1}^{n} \alpha_i$ and is denoted by $\vert \alpha \vert$. To every partition
$\alpha$, we can associate a Young diagram, see \cite{fulton}. The transpose partition is defined to be the partition corresponding to the transpose
of the Young diagram. The transpose of $\alpha$ is denoted by $\alpha'$.
\end{Definition}

Given a partition $\alpha$ and vector bundle $\VectorV$, we  define the following vector bundles
\[      \bigwedge^{\alpha}\VectorV =  \bigotimes_{i} \bigwedge^{\alpha_{i}}\VectorV            \]
\[      \Sym^{\alpha}\VectorV = \bigotimes_{i} \Sym^{\alpha_{i}}\VectorV                       \]
\[      \Divided^{\alpha}\VectorV =  \bigotimes_{i} \Divided^{\alpha_{i}}\VectorV                      \]
where $\Divided^{u}$ is $u$-th divided power representation.

\begin{Definition}
\label{Definition of Schur and Weyl Functors}    
Given a partition $\alpha$ and vector bundle $\VectorV$, the Schur functor $\Schur^{\alpha}\VectorV$ is defined to be the image 
of the composition of the natural symmetrization and antisymmetrization maps \cite[Section I.2,I.3]{Akin-Buchsbaum-Weyman}
\begin{equation}
    \label{Equation Schur Functors}
    \bigwedge^{\alpha}\VectorV \xrightarrow{a} \VectorV^{\vert \alpha \vert}  \xrightarrow{s} \Sym^{\alpha}\VectorV    
\end{equation}
    and the Weyl functor (defined as the coSchur functor in \cite{Akin-Buchsbaum-Weyman}) $\Weyl^{\alpha}\VectorV$ can be defined to be
 the image of the composition of natural maps
 \begin{equation}
    \label{Equation Weyl Functors}
     \Divided^{\alpha}\VectorV   \xrightarrow{s} \VectorV^{\vert \alpha \vert}  \xrightarrow{a}   \bigwedge^{\alpha'}\VectorV.                                                                                                                   
 \end{equation}
\end{Definition}    

We now note down a few properties of the Schur and Weyl functors.

\begin{Lemma}
\label{Lemma Schur and Weyl are vector bundles}    
For a vector bundle $\VectorV$, both the Schur and Weyl functors $\Schur^{\alpha}\VectorV$ and $\Weyl^{\alpha}\VectorV$
are vector bundles. 
\end{Lemma}

\begin{proof}
\cite[Theorem II.2.16]{Akin-Buchsbaum-Weyman} and \cite[Theorem II.3.16]{Akin-Buchsbaum-Weyman}.
\end{proof}

\begin{Lemma}
\label{Lemma Relation between Schur and Weyl}   
For a vector bundle $\VectorV$, there is a natural identification
\[    \Weyl^{\alpha}\!\VectorV = (\Schur^{\alpha}\!\VectorV^{\vee})^{\vee} \]
\end{Lemma}

\begin{proof} This is essentially
\cite[Proposition II.4.1]{Akin-Buchsbaum-Weyman}. We remark here that 
 our definition of Schur functor corresponding to a 
partition $\alpha$ is the Schur functor corresponding to partition $\alpha'$ in \cite{Akin-Buchsbaum-Weyman}.  
\end{proof}

\begin{Theorem}[Universal Cauchy Littlewood Formula]
\label{Theorem Universal Cauchy Littlewood Formula}  
Given vector bundles $\VectorV$ and $\VectorW$, we have a pairing for every partition $\alpha$
	\[       \langle , \rangle_{\alpha} : \bigwedge^{\alpha}\VectorV \otimes \Divided^{\alpha}\VectorW \rightarrow \wedge^{\vert \alpha \vert} (\VectorV \otimes \VectorW)                                          \]
	such that 
	\[     \M_{\alpha}(\VectorV,\VectorW)  = \sum\limits_{\beta \geq \alpha, \vert \alpha \vert = \vert \beta \vert = m} \IM(\langle , \rangle_{\beta})                   \]
	gives a natural filtration on $\bigwedge^{m}(\VectorV \otimes \VectorW)$ whose associated graded object is
	\[ \underset{\vert \alpha \vert = m}{\bigoplus}  \Schur^{\alpha'}\!\VectorV \otimes \Weyl^{\alpha}\!\VectorW . \]
\end{Theorem}

\begin{proof} {This is essentially \cite[Theorem III.2.4]{Akin-Buchsbaum-Weyman}.} We recall a few details.
    For nonnegative integers $p$, we define a pairing inductively
    \[      \langle , \rangle :   \bigwedge^{p}\VectorV \otimes \Divided^{a}\VectorW \rightarrow  \bigwedge^{p} (\VectorV \otimes \VectorW)\]
    following \cite{Akin-Buchsbaum-Weyman}. For $p$ = 1, this pairing is just 
    \[ \langle v , w \rangle = v \otimes w \] 
    For $p$ $>$ 1, this can be defined as 
    \[      \langle v_1 \wedge \hdots \wedge v_p, w_1^{\lambda_1} \hdots w_t^{\lambda_t} \rangle = 
    \sum_{i=1}^{p} \langle v_i , w_1 \rangle \wedge \langle v_1 \wedge \hdots \wedge \hat{v_i} \wedge \hdots \wedge  v_p, w_1^{\lambda_1-1} \hdots w_t^{\lambda_t} \rangle           \]
    This pairing can be extended to give a map for a partition $\alpha$ = $\{ \alpha_{1}, \alpha_{2} \hdots \alpha_{n}\}$ with $\vert \alpha \vert$ = $m$,
    \[       \langle , \rangle_{\alpha} : \bigwedge^{\alpha}\VectorV \otimes \Divided^{\alpha}\VectorW \rightarrow \bigwedge^{\vert \alpha \vert} (\VectorV \otimes \VectorW)                                          \]
    \[  \langle f_1 \otimes \hdots \otimes f_n, g_1 \otimes \hdots \otimes g_n  \rangle_{\alpha} = \langle f_1,g_1 \rangle \wedge \hdots \wedge \langle f_n, g_n \rangle \]
   So we can define a filtration on $\bigwedge^{m}(\VectorV \otimes \VectorW)$ filtered by
     \[     \M_{\alpha}(\VectorV,\VectorW)  = \sum\limits_{\beta \geq \alpha, \vert \beta \vert = \vert \alpha \vert = m} \IM(\langle , \rangle_{\beta})                   \]
   ordered by the lexicographic ordering on the partitions.
    Let
    \[       \M'_{\alpha}(\VectorV,\VectorW)  = \sum\limits_{\beta > \alpha, , \vert \beta \vert = \vert \alpha \vert = m} \IM(\langle , \rangle_{\beta})                                                                                \]
    then the pairing $\langle , \rangle_\alpha$ induces a natural map
    \[ \Schur^{\alpha'}\!\VectorV \otimes \Weyl^{\alpha}\!\VectorW \rightarrow \M_{\alpha}(\VectorV,\VectorW)/\M'_{\alpha}(\VectorV,\VectorW) \] 
    This map is locally an isomorphism by \cite[Theorem III.2.4]{Akin-Buchsbaum-Weyman} thus proving that the associated graded to this filtration is isomorphic to 
    \[ \underset{\vert \alpha \vert = m}{\bigoplus}  \Schur^{\alpha'}\!\VectorV \otimes \Weyl^{\alpha}\!\VectorW . \]
\end{proof}

\section{The Derived Category of a Grassmannian}
\label{s:grassmanian}
   
In this section we will construct a semi-orthogonal decomposition of the bounded derived category of a Grassmannian.
Such a construction has appeared in \cite{efimov}. We will describe here a different approach to this theorem.

\begin{Notation}
    \label{Notation Ordering on Partitions }
    Let $<$ denote the standard lexiographic ordering on partitions. $B_{k,n}$ denote partitions $\alpha$ corresponding to Young diagrams
    with no more than $k$ rows and no more than $n - k$ columns with total ordering $\prec$ such that :
    \begin{enumerate}
        \item If $\vert \alpha \vert$  $<$ $\vert \beta \vert$, then $\alpha$ $\prec$ $\beta$.
        \item If $\vert \alpha \vert$  $=$ $\vert \beta \vert$ and $\beta$ $<$ $\alpha$ in lexicographic ordering  then $\alpha$ $\prec$ $\beta$.
    \end{enumerate}
    and let $\overline{B_{k,n}}$ denote the same set with the opposite ordering.
\end{Notation}
 
The first result in the direction that we are headed is due to Kapranov.

\begin{Theorem}
    \label{Theorem Kapranov} 
    Consider a Grassmannian $\Gr(k,n)_{K}$ over a field $K$ of characteristic zero.
	Then the  set $\{\Schur^{\alpha}\!\Tautological \}$ for $\alpha$ in $B_{k,n}$ 
    is a strongly exceptional sequence on the Grassmannian $\Gr(k,n)_{K}$.
\end{Theorem}

\begin{proof}
    See \cite{Kapranov1}.
\end{proof}

For general characteristic, Buchweitz, Leuschke  and Van Der Bergh proved in \cite{Grassmannian-arbitrary-characteristic}:

\begin{Theorem}
    \label{Theorem Grassmannian in arbitrary characteristic}   
    $\{\Schur^{\alpha}\!\Tautological^{\vee} \}$ for $\alpha$ in $B_{k,n}$ 
    is an exceptional sequence on the Grassmannian $\Gr(k,n)_{K}$.
\end{Theorem}

In particular, they showed that

\begin{Theorem}
	\label{Theorem cohomological properties of Dual Tautological}
    Let $\alpha$, $\beta$ be in $B_{k,n}$. 
    \begin{enumerate}
	\item \[\RHom_{\mathcal{O}_{\Gr(k,n)_{K}}} (\Schur^{\alpha}\!\Tautological^{\vee},\Schur^{\beta}\!\Tautological^{\vee}) = 0       \]
	for $\beta$ $\prec$ $\alpha$.
	\item  \[\RHom_{\mathcal{O}_{\Gr(k,n)_{K}}} (\Schur^{\alpha}\!\Tautological^{\vee},\Schur^{\alpha}\!\Tautological^{\vee}) = K.       \]
\end{enumerate}
\end{Theorem}

\begin{proof}
    {This is deduced from \cite[Theorem 1.4]{Grassmannian-arbitrary-characteristic}.} Note that, what is $\Tautological^\vee$ in our notation is 
    denoted ${\mathcal Q}$ in \emph{loc. cit.}
\end{proof}

We will make use of the following lemma which follows from the above Theorem.

\begin{Lemma}
    \label{Lemma Cohomology of Weyl functors}
    In the above situation,
    let $\alpha$, $\beta$ $\in$ $\overline{B_{k,n}}$. Then 
    \begin{enumerate}
	\item \[\RHom_{\mathcal{O}_{\Gr(k,n)_{K}}} (\Weyl^{\alpha}\!\Tautological,\Weyl^{\beta}\!\Tautological) = 0       \]
	for $\beta$ $\prec$ $\alpha$.
	\item  \[\RHom_{\mathcal{O}_{\Gr(k,n)_{K}}} (\Weyl^{\alpha}\!\Tautological,\Weyl^{\alpha}\!\Tautological) = K.       \]
\end{enumerate}
\end{Lemma}

\begin{proof}
 Since all the sheaves involved are locally free (Lemma \ref{Lemma Schur and Weyl are vector bundles})
\begin{align*}
    \RHomi_{\mathcal{O}_{\Gr(k,n)_{K}}} (\Weyl^{\alpha}\!\Tautological,\Weyl^{\beta}\!\Tautological) &= \Cohomology^i (\mathcal{H}om(\Weyl^{\alpha}\!\Tautological,\Weyl^{\beta}\!\Tautological)) \\
                                                                                                        &= \Cohomology^i (\Schur^{\alpha}\!\Tautological^{\vee} \otimes \Weyl^{\beta}\!\Tautological) \\
                                                                                                        &= \RHomi_{\mathcal{O}_{\Gr(k,n)_{K}}} (\Schur^{\beta}\!\Tautological^{\vee},\Schur^{\beta}\!\Tautological^{\vee}) \\
\end{align*}
We have used Lemma \ref{Lemma Relation between Schur and Weyl} extensively throughout this computation.
\end{proof}

\begin{Theorem}
    \label{Theorem Derived Category of Grassmannian over algebraic stack for Dqc}
    Let $\VectorV$ be a rank $n$ vector bundle on an algebraic stack $X$ and let
    \[     \pi :  \Gr(k,\VectorV)_{X} \rightarrow X                                         \]
    be the projection.
    We denote by $\Phi_\alpha$ the finite Fourier-Mukai functor induced by the finite Fourier-Mukai transform 
    $(\pi, 1, \Weyl^\alpha \Tautological)$ so that 
    \begin{eqnarray*}
        \Phi_\alpha :& \Dqc(X) \rightarrow \Dqc(\Gr(k,\VectorV)_{X})      \\
        \Phi_\alpha(\E) =&  \Left\!\pi^*(\E) \otimes \Weyl^{\alpha}\!\Tautological. 
    \end{eqnarray*}
    Then there is  a semiorthogonal decomposition of $\Dqc(\Gr(k,\VectorV)_{X})$ into right admissible categories
    \[      \Dqc(\Gr(k,\VectorV)_{X}) = \langle \{ \IM(\Phi_{\alpha}) \}_{\alpha \in \overline{B_{k,n}}}   \rangle                    \]                                                  
\end{Theorem}

\begin{proof} The
 $\Phi_\alpha$ are fully faithful Fourier Mukai functors by Lemma \ref{Lemma Cohomological properties of Phi}. The same lemma also
 implies that their essential images are semiorthogonal. Moreover, this is a semiorthogonal decomposition by conservative descent
(Theorem \ref{Lemma Geometric Conservative Descent Theorem}) since their base change to an affine scheme gives a semiorthogonal
decomposition by Lemma \ref{Lemma fullness of the semiorthogonal decomposition}.
\end{proof}

\begin{Corollary}
    \label{Corollary Derived Category of Grassmannian over algebraic stack for D perfect}
    In the setting of the theorem, 
there are semiorthogonal decompositions of the derived categories $\D_p(\Gr(k,\E)_{X})$, $\D^{lb}_{pc}(\Gr(k,\E)_{X})$ and
 $\D_{sg}(\Gr(k,\E)_{X})$ into right
admissible subcategories given by essential images of the induced fully faithful functors ( for $\alpha \in \overline{B_{k,n}}$)
\[   \Phi_{\alpha}^p : \D_p(X) \rightarrow \D_p(\Gr(k,\E)_{X})  \] 
\[   \Phi_{\alpha}^{pc} :  \D^{lb}_{pc}(X) \rightarrow \D^{lb}_{pc}(\Gr(k,\E)_{X}) \]
\[   \Phi_{\alpha}^{sg} :  \D_{sg}(X) \rightarrow \D_{sg}(\Gr(k,\E)_{X})                                                              \]                             
respectively.                                                        
\end{Corollary}

\begin{proof}
Theorem \ref{Theorem Derived Category of Grassmannian over algebraic stack for Dqc} and 
Lemma \ref{Lemma semiorthogonal decomposition of Dqc gives semiorthogonal decomposition of other derived categories}.
\end{proof}

\section{Ingredients of the Proof}

\subsection{Cohomology}

\begin{Lemma}
\label{Lemma Pushforward of Hom of Weyl with same partition alpha}
Consider a fixed affine scheme $\Spec(R)$.
Let $\Tautological$ be the tautological bundle over the Grassmannian
\[ \pi :\Gr(k,n)_R \rightarrow \Spec(R). \]
Then, for any $\alpha$ in $\overline{B_{k,n}}$,
\[  \Right\!\pi_{*}(\RsheafHom(\Weyl^{\alpha}\!\Tautological, \Weyl^{\alpha}\!\Tautological)) \simeq \mathcal{O}_{R} .                    \]
\end{Lemma} 

\begin{proof}
For any point $a$ in $\Spec(R)$, we have by Lemma \ref{Lemma Cohomology of Weyl functors} 
\[      \Cohomology^i(\Gr(k,n)_{R}|_a, \RsheafHom(\Weyl^{\alpha}\!\Tautological, \Weyl^{\alpha}\!\Tautological)|_a) = 0 \] 
for $i$ $>$ $0$ and
\[      \Cohomology^0(\Gr(k,n)_{R}|_a, \RsheafHom(\Weyl^{\alpha}\!\Tautological, \Weyl^{\alpha}\!\Tautological)|_a) = k(a). \] 
Hence by \cite[Theorem III.12.11]{Hartshorne}
\[ \Right^i\!\pi_{*}(\RsheafHom(\Weyl^{\alpha}\!\Tautological, \Weyl^{\alpha}\!\Tautological)) \otimes k(a) \simeq \Cohomology^{i}(\RsheafHom(\Weyl^{\alpha}\!\Tautological, \Weyl^{\alpha}\!\Tautological)|_{a})              \]
which is $0$ for $i$ $>$ $0$ and $k(a)$ for $i$ = $0$. The same theorem gives us that
$\Right^i\!\pi_{*}(\RsheafHom(\Weyl^{\alpha}\!\Tautological, \Weyl^{\alpha}\!\Tautological))$ is locally free. This further implies
\[        \Right^i\!\pi_{*}(\RsheafHom(\Weyl^{\alpha}\!\Tautological, \Weyl^{\alpha}\!\Tautological))  = 0                                                                                                \]
for $i$ $>$ 0 by Nakayama's lemma. Hence,
\[    \Cohomology^{0}( \Spec(R), \Right^0\!\pi_{*}(\RsheafHom(\Weyl^{\alpha}\!\Tautological, \Weyl^{\alpha}\!\Tautological))) \simeq    \Cohomology^{0}( \Gr(k,n)_{R}, \mathcal{H}om(\Weyl^{\alpha}\!\Tautological, \Weyl^{\alpha}\!\Tautological))                                         .     \]
furnishing a global nonvanishing section of $\Right^0\!\pi_{*}(\RsheafHom(\Weyl^{\alpha}\!\Tautological, \Weyl^{\alpha}\!\Tautological))$
corresponding to the global non vanishing section on $\mathcal{H}om(\Weyl^{\alpha}\!\Tautological, \Weyl^{\alpha}\!\Tautological)$
given by the identity map. Hence, $\Right^0\!\pi_{*}(\RsheafHom(\Weyl^{\alpha}\!\Tautological, \Weyl^{\alpha}\!\Tautological))$ is an invertible sheaf 
with a global nonvanishing section and thus isomorphic to the structure sheaf $\mathcal{O}_{R}$. 
\end{proof}

\begin{Lemma}
  \label{Lemma Semi Orthogonality of Fourier Mukai}
  Let  $\beta$ and  $\alpha$ be elements of $\overline{B_{k,n}}$ satisfying $\beta$ $\prec$ 
  $\alpha$. Consider the Grassmannian over an affine scheme
  \[     \pi :     \Gr(k,n)_R \rightarrow   \Spec(R)                                      \]
  and recall the Fourier-Mukai functors $\Phi_\alpha$ defined in Theorem \ref{Theorem Derived Category of Grassmannian over algebraic stack for Dqc}.
  Then for $\E$ $\in$ $\IM(\Phi_{\alpha})$, $\F$ $\in$ $\IM(\Phi_{\beta})$
  \[            \RHom(\E,\F) = 0  .         \]   
\end{Lemma}

\begin{proof}
We compute 
\begin{align*}
	\RHom(\Phi_{\alpha}(\E'),\Phi_{\beta}(\F')) 
	&= \RHom(\Left\!\pi^{*}(\E') \otimes \Weyl^{\alpha}\!\Tautological, \Left\!\pi^{*}(\F') \otimes \Weyl^{\beta}\!\Tautological)\\
    &\simeq \RHom(\Left\!\pi^{*}(\E'), \RsheafHom( \Weyl^{\alpha}\!\Tautological, \Left\!\pi^{*}(\F') \otimes \Weyl^{\beta}\!\Tautological))\\
    &\simeq \RHom(\Left\!\pi^{*}(\E'), \Left\!\pi^{*}(\F') \otimes \RsheafHom(\Weyl^{\alpha}\!\Tautological, \Weyl^{\beta}\!\Tautological))\\
    &\simeq \RHom(\E', \Right\!\pi_{*}( \Left\!\pi^{*}(\F') \otimes \RsheafHom(\Weyl^{\alpha}\!\Tautological, \Weyl^{\beta}\!\Tautological )))\\
    &\simeq \RHom(\E', \F' \otimes \Right\!\pi_{*}(\RsheafHom(\Weyl^{\alpha}\!\Tautological, \Weyl^{\beta}\!\Tautological )))
\end{align*} 
where we have used \cite[Proposition 2.6.1(ii)]{Derived-Functors-Grothendieck-Duality} for the Hom-Tensor adjunction, 
\cite[Proposition 3.2.3(i)]{Derived-Functors-Grothendieck-Duality} for pushforward pullback adjunction and 
\cite[Proposition 3.9.4]{Derived-Functors-Grothendieck-Duality} for the projection formula. Now applying 
\cite[Theorem III.12.11]{Hartshorne} in a similar way as in lemma \ref{Lemma Pushforward of Hom of Weyl with same partition alpha}, 
we have
\[ \Right^i\!\pi_{*}(\RsheafHom(\Weyl^{\alpha}\!\Tautological, \Weyl^{\beta}\!\Tautological)) \otimes k(a) \simeq \Cohomology^{i}(\RsheafHom(\Weyl^{\alpha}\!\Tautological, \Weyl^{\beta}\!\Tautological)|_{a})   = 0           \]
and $\Right^i\!\pi_{*}(\RsheafHom(\Weyl^{\alpha}\!\Tautological, \Weyl^{\beta}\!\Tautological))$ is locally free. Thus $\Right^i\!\pi_{*}(\RsheafHom(\Weyl^{\alpha}\!\Tautological, \Weyl^{\beta}\!\Tautological))$ $\simeq$ 0
by Nakayama. From the previous computation,
\[ \RHom(\Phi_{\alpha}(\E'),\Phi_{\beta}(\F')) = 0. \]
\end{proof}

\begin{Lemma}
    \label{Lemma Fourier Mukai are fully faithful}
    With the notation as in Lemma \ref{Lemma Semi Orthogonality of Fourier Mukai}, we have
    \[    \RHom(\Phi_{\alpha}(\E'),\Phi_{\alpha}(\F')) = \RHom(\E',\F')                                                                    \]
    for $\E'$ and $F'$ in $\Dqc(\Spec(R))$.
\end{Lemma}

\begin{proof}
We perform a similar computation as in the previous lemma
    \begin{align*}
        \RHom(\Phi_{\alpha}(\E'),\Phi_{\alpha}(\F'))
                   &= \RHom(\Left\!\pi^{*}(\E') \otimes \Weyl^{\alpha}\!\Tautological, \Left\!\pi^{*}(\F') \otimes \Weyl^{\alpha}\!\Tautological)\\
                   &\simeq \RHom(\Left\!\pi^{*}(\E'), \Left\!\pi^{*}(\F') \otimes \RsheafHom(\Weyl^{\alpha}\!\Tautological, \Weyl^{\alpha}\!\Tautological))\\
                   &\simeq \RHom(\E', \Right\!\pi_{*}( \Left\!\pi^{*}(\F') \otimes \RsheafHom(\Weyl^{\alpha}\!\Tautological, \Weyl^{\beta}\!\Tautological )))\\
                   &\simeq \RHom(\E', \F' \otimes \Right\!\pi_{*}(\RsheafHom(\Weyl^{\alpha}\!\Tautological, \Weyl^{\beta}\!\Tautological )))\\
                   &\simeq \RHom(\E',\F'). 
    \end{align*}
    where the last equality follows from lemma \ref{Lemma Pushforward of Hom of Weyl with same partition alpha}.
\end{proof}

We now collect all these properties into one lemma

\begin{Lemma}
    \label{Lemma Cohomological properties of Phi}
    With notation as in Theorem \ref{Theorem Derived Category of Grassmannian over algebraic stack for Dqc}, the $\Phi_{\alpha}$
    are finite Fourier-Mukai functors in the sense of Definition \ref{Definition Fourier Mukai Transform}. Moreover, they are
    fully faithful and their essential images $\{\IM(\Phi_{\alpha}) \}_{\alpha \in \overline{B_{k,n}}}$ with ordering as in
    $\overline{B_{k,n}}$ is a semiorthogonal sequence.
\end{Lemma}

\begin{proof}
We have that $(\pi, \Identity, \Weyl^\alpha\Tautological)$ 
is a finite Fourier-Mukai transform since $\pi$ is proper, flat and of finite presentation and therefore perfect. We denote by $\Phi_{\alpha}$ the 
corresponding Fourier-Mukai functor. They are fully faithful by applying geometric conservative descent as in theorem
\ref{Lemma Geometric Conservative Descent Theorem} to  lemma
\ref{Lemma Fourier Mukai are fully faithful}. Similarly, Semiorthogonality follows from lemma
\ref{Lemma Semi Orthogonality of Fourier Mukai} and \ref{Lemma Geometric Conservative Descent Theorem}.
\end{proof}

\subsection{Generation}
In this subsection, we use the standard resolution of the diagonal argument to produce generators of the derived category of a 
Grassmannian. This result is well known, see for example \cite[Remark 7.6]{efimov}. 

We have the the following pullback diagram:

\begin{equation}
    \label{Diagram Product of Grassmannians}
    \begin{tikzcd}
        \Gr(k,n)_{R} \times \Gr(k,n)_{R} \arrow[d, "p"'] \arrow[r, "q"] & \Gr(k,n)_{R} \arrow[d, "\pi"] \\
        \Gr(k,n)_{R} \arrow[r, "\pi"]                                                 & 
        \Spec(R)             \\
        \end{tikzcd}
\end{equation}
for the Grassmannian over an affine scheme $\Spec(R)$.
\begin{Notation}
	For sheaves $\E$ and $\F$ on $\Gr(k,n)_{R}$, we denote by  $\E \boxtimes \F$ the sheaf $p^{*}(\E) \otimes q^{*}(\F)$ on 
    $\Gr(k,n)_{R} \times \Gr(k,n)_{R}$.
\end{Notation}

We define a section $s$ of $\Tautological^{\vee} \boxtimes \Quotient$ by the composition of the natural maps
\[ p^*\Tautological \rightarrow p^*\mathcal{O}_{\Gr(k,n)_{R}}^{\oplus n} \simeq q^*\mathcal{O}_{\Gr(k,n)_{R}}^{\oplus n} \rightarrow q^*\Quotient.   \]    

\begin{Lemma}
    \label{Lemma image of dual of the section}
        The image of $s^{\vee}$ is surjective onto the ideal sheaf of the diagonal.
\end{Lemma}

\begin{proof}
We have to prove that the sequence
\[   \Tautological \boxtimes \Quotient^{\vee}  \xrightarrow[]{s^{\vee}}   \mathcal{O}_{\Gr(k,n)_{R} \times \Gr(k,n)_{R}}      \rightarrow \Delta_* \mathcal{O}_{\Gr(k,n)_{R}} \rightarrow 0 \] 
on $\Gr(k,n)_{R} \times \Gr(k,n)_{R}$ is exact where
\[         \Delta : \Gr(k,n)_{R} \rightarrow  \Gr(k,n)_{R} \times \Gr(k,n)_{R} \]                                            
is the diagonal embedding. It is enough to check this at the level of stalks. We first check at a point $(x,x)$ in 
$\Gr(k,n)_{R} \times \Gr(k,n)_{R}$. This point $(x,x)$ is contained in an open affine space $U_I \times U_I$ and furthermore we can 
assume $I$ = $[n-k]$. By Remark \ref{Remark Description of Open Sets of Grassmannian}, 
\[ U_I \times U_I = k[\{x_{1,ij}\}_{\substack{1 \leq i \leq n-k \\ n-k < j \leq n}},\{x_{2,ij}\}_{\substack{1 \leq i \leq n-k \\ n-k < j \leq n}}] \]
and $p^*\Tautological$ is generated by elements $\{s_j\}_{n-k < j \leq n}$
\[    s_j = e_j +  \sum_{i=1}^{n-k}  x_{1,ji} e_i.                                        \]
On the other hand $q^*\Quotient$ is generated by elements $\{q^*_i\}_{1 \leq i \leq n-k}$
\[    t^*_i =  e^*_i -  \sum_{j=n-k+1}^n x_{2,ji} e^*_j.                            \]
Hence image of $s^{\vee}$ is generated by 
\[      t^*_i(s_j) =  (e^*_i -  \sum_{j=n-k+1}^n x_{2,ji} e^*_j)(e_j +  \sum_{i=1}^{n-k}  x_{1,ji} e_i) = x_{1,ji} - x_{2,ji}. \]                                                                      
Thus image of $s^{\vee}$ is the ideal sheaf of the diagonal 
\[           \mathcal{I}_{\Delta} = \{ x_{1,ji} - x_{2,ji}  \}_{\substack{1 \leq i \leq n-k \\ n-k < j \leq n}}.          \]
In the case of a point $*$ = $(x,y)$ with $x \neq y$, we claim $s(x,y)$ is non zero. In this case,
$s^{\vee}$ is surjective onto $(\mathcal{O}_{\Gr(k,n)_{R} \times \Gr(k,n)_{R}})_{(x,y)}$ and we are done since
$(\Delta_* \mathcal{O}_{\Gr(k,n)_{R}})_{x,y}$ = 0. To prove the claim, we assume to the contrary $s(x,y)$ = 0. In that
case we have an exact sequence by the definition of $s$
\[     0 \rightarrow p^*(\Tautological)_{(x,y)} \rightarrow q^*(\mathcal{O}^{\oplus n}_{\Gr(k,n)_{R}})_{(x,y)} \rightarrow q^{*}(\Quotient)_{(x,y)} \rightarrow 0\] 
which fits into a commutative diagram  
\[
\begin{tikzcd}[row sep=2em, column sep=3em]
    0 \arrow[r] & p^*(\Tautological)_{(x,y)} \arrow[r] \arrow[d] & p^*(\mathcal{O}^{\oplus n}_{\Gr(k,n)_{R}})_{(x,y)} \arrow[r] \arrow[d] & p^{*}(\Quotient)_{(x,y)} \arrow[r] \arrow[d, dotted] & 0 \\
    0 \arrow[r] & p^*(\Tautological)_{(x,y)} \arrow[r] & q^*(\mathcal{O}^{\oplus n}_{\Gr(k,n)_{R}})_{(x,y)} \arrow[r] & q^{*}(\Quotient)_{(x,y)} \arrow[r] & 0
\end{tikzcd}
\]
where the rightmost vertical arrow is an isomorphism by the five lemma. We replace $\Gr(k,n)_R \times \Gr(k,n)_{R}$ by the  diagram \ref{Diagram Product of Grassmannians} by
\begin{equation}
    \label{Diagram Product of Grassmannians and points}
    \begin{tikzcd}
        * \arrow[d, "p"'] \arrow[r, "q"] & \Gr(k,n)_{R} \arrow[d, "\pi"] \\
        \Gr(k,n)_{R} \arrow[r, "\pi"]                                                 & 
        \Spec(R)             \\
        \end{tikzcd}
\end{equation}
and over $*$ = $(x,y)$, we have the commutative diagram
\[
\begin{tikzcd}[row sep=2em, column sep=3em]
    p^*(\mathcal{O}^{\oplus n}_{\Gr(k,n)_{R}}) \arrow[r, two heads] \arrow[d, "\simeq"'] & p^{*}(\Quotient) \arrow[d, "\simeq"] \\
    q^*(\mathcal{O}^{\oplus n}_{\Gr(k,n)_{R}}) \arrow[r, two heads] & q^{*}(\Quotient)
\end{tikzcd}
\]
Since $p^{*} \VectorV$ of a vector bundle $\VectorV$ is just the fibre $\VectorV|_x$  at $x$ and $q^*\VectorV$ = $\VectorV|_y$, we 
have the commutative diagram
\[
\begin{tikzcd}[row sep=2em, column sep=3em]
    \mathcal{O}^{\oplus n}_{\Gr(k,n)_{R}}|_x \arrow[r, two heads] \arrow[d, "\simeq"'] & \Quotient|_x \arrow[d, "\simeq"] \\
    \mathcal{O}^{\oplus n}_{\Gr(k,n)_{R}}|_y \arrow[r, two heads] & \Quotient|_y
\end{tikzcd}
\]
which is absurd in case $x$ $\neq$ $y$ by the definition of the Grassmannian functor.
\end{proof}

The following lemma is stated in \cite[remark 7.6]{efimov}.

\begin{Lemma}
    \label{Lemma fullness of the semiorthogonal decomposition}
    With notation as in lemma \ref{Lemma Semi Orthogonality of Fourier Mukai}, the smallest strictly full 
    triangulated subcategory of $\Dqc(\Gr(k,n)_{R})$ containing the set $\{ \IM(\Phi'_{\alpha}) \}_{\alpha \in \overline{B_{k,n}}}$ is
    $\Dqc(\Gr(k,n)_{R})$ itself.
\end{Lemma}

\begin{proof}
    We obtain a Koszul resolution for $s^{\vee}$ by lemma \ref{Lemma image of dual of the section} given by
    \begin{align*}
        0 \rightarrow  \bigwedge^{k(n-k)} (\Tautological \boxtimes \Quotient^{\vee}) \rightarrow \hdots \rightarrow  &\bigwedge^{3} (\Tautological \boxtimes \Quotient^{\vee})
          \rightarrow   \bigwedge^{2} (\Tautological \boxtimes \Quotient^{\vee})\\              &\rightarrow        \bigwedge^{1} (\Tautological \boxtimes \Quotient^{\vee})
          \rightarrow   \mathcal{O}_{\Gr(k,n)_{R} \times \Gr(k,n)_{R}}      \rightarrow \Delta_* \mathcal{O}_{\Gr(k,n)_{R}} \rightarrow 0.     
    \end{align*}
        Since
     \[\bigwedge^{i} (\Tautological \boxtimes \Quotient^{\vee}) \simeq (\bigwedge^{i} (\Tautological^{\vee} \boxtimes \Quotient))^{\vee} \]
    we can write the above long exact sequence as
    \begin{equation}
        \label{Equation Long Exact Sequence Rearranged Resolution of Diagonal}
        \begin{aligned}
            0 \rightarrow  (\bigwedge^{k(n-k)} (\Tautological^{\vee} \boxtimes \Quotient))^{\vee} \rightarrow \hdots \rightarrow  &\bigwedge^{3} (\Tautological \boxtimes \Quotient^{\vee})
                        \rightarrow   (\bigwedge^{2} (\Tautological^{\vee} \boxtimes \Quotient))^{\vee}\\              &\rightarrow        (\bigwedge^{1} (\Tautological^{\vee} \boxtimes \Quotient))^{\vee}
                        \rightarrow   \mathcal{O}_{\Gr(k,n)_{R} \times \Gr(k,n)_{R}}      \rightarrow \Delta_* \mathcal{O}_{\Gr(k,n)_{R}} \rightarrow 0. 
        \end{aligned}
    \end{equation}
    On the other hand, given an object $\E$ in $\Dqc(\Gr(k,n)_{R})$, we have an exact functor
    \begin{align*}
        \Psi : &\Dqc(\Gr(k,n)_{R} \times \Gr(k,n)_{R}) \rightarrow \Dqc(\Gr(k,n)_{R})\\
               &\F                             \mapsto \Psi_{\F}(\E) := \Right\!p_*(\Left\!q^*\E \otimes_{\Left} \F ).  
    \end{align*}
    Let $\Line$ be an ample line bundle on $\Gr(k,n)_{R}$. Then applying $\Psi$ to the long exact sequence 
    equation \ref{Equation Long Exact Sequence Rearranged Resolution of Diagonal}, we deduce that
    $\Line$ $\simeq$ $\Psi_{\Delta_* \mathcal{O}_{\Gr(k,n)_{R}}}(\Line)$ is contained in the smallest strictly full triangulated category
    containing $\{\Psi_{(\bigwedge^{i} (\Tautological^{\vee} \boxtimes \Quotient))^{\vee}}(\Line)\}_{i \geq 0}$. Utilizing 
    theorem \ref{Theorem Universal Cauchy Littlewood Formula}, we get sequences of short exact sequences for each $i$,
    \begin{gather*}
        0 \rightarrow G_1(\Tautological^{\vee},\Quotient) \rightarrow M_{\alpha_1}(\Tautological^{\vee},\Quotient) \rightarrow G_2(\Tautological^{\vee},\Quotient) \rightarrow 0.\\ 
        0 \rightarrow M_{\alpha_1}(\Tautological^{\vee},\Quotient) \rightarrow M_{\alpha_2}(\Tautological^{\vee},\Quotient) \rightarrow G_3(\Tautological^{\vee},\Quotient) \rightarrow 0.\\ 
        \hdots\\ 
        0 \rightarrow M_{\alpha_p}(\Tautological^{\vee},\Quotient) \rightarrow \bigwedge^{i} (\Tautological^{\vee} \boxtimes \Quotient) \rightarrow G_{p+1}(\Tautological^{\vee},\Quotient) \rightarrow 0. 
    \end{gather*}
    where each $G_i(\Tautological^{\vee},\Quotient)$ is isomorphic to $\Schur^{\alpha}\!\Tautological^{\vee} \boxtimes \Weyl^{\alpha'}\!\Quotient$ 
    for some $\alpha$. We can dualize this short exact sequences since the rightmost terms are flat (lemma \ref{Lemma Schur and Weyl are vector bundles})
    to obtain another collection of short exact sequences 
    \begin{gather*}
        0 \rightarrow (G_2(\Tautological,\Quotient))^{\vee} \rightarrow (M_{\alpha_1}(\Tautological,\Quotient))^{\vee} \rightarrow (G_1(\Tautological,\Quotient))^{\vee} \rightarrow 0.\\ 
        0 \rightarrow (G_3(\Tautological,\Quotient))^{\vee} \rightarrow (M_{\alpha_2}(\Tautological,\Quotient))^{\vee} \rightarrow (M_{\alpha_1}(\Tautological,\Quotient))^{\vee} \rightarrow 0.\\ 
        \hdots\\ 
        0 \rightarrow (G_{p+1}(\Tautological,\Quotient))^{\vee} \rightarrow \left(\bigwedge^{i} (\Tautological \boxtimes \Quotient)\right)^{\vee} \rightarrow (M_{\alpha_p}(\Tautological,\Quotient))^{\vee} \rightarrow 0.  
    \end{gather*}
    where each $G_i(\Tautological^{\vee},\Quotient)^{\vee}$ is isomorphic to $(\Schur^{\alpha}\!\Tautological^{\vee})^{\vee} \boxtimes (\Weyl^{\alpha'}\!\Quotient)^{\vee}$ 
    $\underset{\ref{Lemma Relation between Schur and Weyl}}{\simeq}$ $\Weyl^{\alpha}\!\Tautological \boxtimes \Schur^{\alpha'}\!\Quotient^{\vee}$ for some $\alpha$.
    Applying $\Psi$ to this collection, we deduce $\{\Psi_{(\bigwedge^{i} (\Tautological^{\vee} \boxtimes \Quotient))^{\vee}}(\Line)\}_{i \geq 0}$  and hence, $\Line$
    is contained in the smallest strictly full triangulated subcategory generated by 
    $\{\Psi_{\Weyl^{\alpha}\!\Tautological \boxtimes \Schur^{\alpha'}\!\Quotient^{\vee}}(\Line)\}_{\alpha \in \overline{B_{k,n}}}$. 
    We can restrict $\alpha$ to have fewer than $k$ rows and $n-k$ columns since $\Weyl^{\alpha}\!\Tautological$ vanishes when 
    $\alpha$ exceeds $k$ rows, and $\Schur^{\alpha'}\!\Quotient^{\vee}$ vanishes when $\alpha$ exceeds $n-k$ columns owing to their 
    respective dimensions. We now compute the term $\Psi_{\Weyl^{\alpha}\!\Tautological \boxtimes \Schur^{\alpha'}\!\Quotient^{\vee}}(\Line)$.
    \begin{align*}
        \Psi_{\Weyl^{\alpha}\!\Tautological \boxtimes \Schur^{\alpha'}\!\Quotient^{\vee}}(\Line) &= \Right\!p_*(\Left\!q^*\Line \otimes_{\Left} (\Weyl^{\alpha}\!\Tautological \boxtimes \Schur^{\alpha'}\!\Quotient^{\vee}) )\\
                                                                                                 &= \Right\!p_*(\Left\!q^*\Line \otimes_{\Left} \Left\!p^*\Weyl^{\alpha}\!\Tautological \otimes_{\Left} \Left\!q^*\Schur^{\alpha'}\!\Quotient^{\vee})\\ 
                                                                                                 &= \Weyl^{\alpha}\!\Tautological \otimes \Right\!p_* \Left\!q^*(\Line \otimes \Schur^{\alpha'}\!\Quotient^{\vee})\\
                                                                                                 &= \Weyl^{\alpha}\!\Tautological \otimes \Left\!\pi^*\Right\!\pi_*(\Line \otimes \Schur^{\alpha'}\!\Quotient^{\vee})\\
                                                                                                 &= \Phi'_{\alpha}(\Right\!\pi_*(\Line \otimes \Schur^{\alpha'}\!\Quotient^{\vee}))
    \end{align*}
    where we have used projection formula and flat base change with respect to diagram \ref{Diagram Product of Grassmannians}.
    On the other hand, $\Phi'_{\alpha}$ are fully faithful Fourier-Mukai Functors (lemma \ref{Lemma Cohomological properties of Phi})
    and their essential images are a semiorthogonal sequence (lemma \ref{Lemma Cohomological properties of Phi})
    of admissible subcategories (lemma \ref{Lemma Essential Images of Fourier Mukai Functors are admissible}). We have just 
    proved that $\Line$ is in the smallest strictly full triangulated subcategory containing $\{ \IM(\Phi'_{\alpha}) \}_{\alpha \in \overline{B_{k,n}}}$.
    So we can conclude by using lemma \ref{Lemma Generator contained in Semiorthogonal Sequence implies Semiorthogonal Decomposition}  since $\Line$ is a generator of $\Dqc(\Gr(k,n)_{R})$ 
    (lemma \ref{Lemma Ample Line Bundle Generator}). 
    \end{proof}

    \section{Generalized Severi Brauer Schemes} \label{s:gsb}

    \begin{Remark}\label{r:faithful}
        In what follows, we will make use of the fact that a Grassmannian $\Gr(k,n)_X$ admits a faithful action of ${\rm PGL}_n$. This is easily proved
        by looking at the action of ${\rm GL}_n$ on subbundles of ${\mathcal{O}_X}^n$.
    \end{Remark}

    \begin{Definition}
        \label{Definition Generalised Severi Brauer }
        An algebraic stack $Y$ over $X$ is called \emph{generalised Severi-Brauer stack} if it is fppf locally isomorphic to $\Gr(k,V)$ for 
        some $k$ and vector bundle $V$ on a fppf cover of $X$. This means that there is a 2-Cartesian diagram
        $$\begin{tikzcd}
            \Gr(k,V) \ar[r] \ar[d] & Y \ar[d] \\
            X'\ar[r] & X,
        \end{tikzcd}$$
        where $X'\to X$ is fppf.
    \end{Definition}
        
   We will use the notation $\SB(k,n)_X\to X$ for a generalised Severi-Brauer stack over $X$. The subscript $X$ is often dropped from the notation
   when it is clear from the context what the base stack $X$ is. 
   
   Given such a stack, we can form its \emph{gerbe of trivialisations}. This is algebraic stack with a morphism to $X$ that we denote by 
   $G(\SB(k,n)_X)\rightarrow X$. Given a point $t:\Spec(R)\rightarrow X$, the groupoid over $t$ has as objects pairs $(V,f)$ where 
   $V$ is a locally free sheaf on $\Spec(R)$ and 
   \[ f : (\SB(k,n))_R \rightarrow \Gr(k,V^{\vee})              \] 
 is an isomorphism. A morphism $( V, f)\rightarrow (V', f')$ is an isomorphism $V \rightarrow V'$ such that the obvious diagram commutes.
 In view of \ref{r:faithful}, this is a $\Gm$-gerbe. 

 It follows from the  definition of $G(\SB(k,n)_X)$ that there is a universal locally free sheaf $V$ of rank $n$ on $G(\SB(k,n)_X)$ and a 2-Cartesian diagram
 \begin{equation}\label[*]{Diagram Gerbe of Trivialisation}
   \begin{tikzcd}
    \Gr(k,V) \ar[d,"\pi'"]\ar[r,"q'"] & \SB(k,n)_X \ar[d,"\pi"]\\ 
    G(\SB(k,n)_X) \ar[r,"q"]& X.
   \end{tikzcd}
 \end{equation}
 It follows that $\Gr(k,V) \to \SB(k,n)_X$ is a $\Gm$-gerbe.

 In what follows, we will make use of twisted sheaves and gerbes corresponding to cohomology classes. See, for example, \cite{dejong} and \cite{lieblich}.

    Let 
   \[ q : G \rightarrow X \]
   be a $\mathbb{G}_m$-gerbe  on $X$ corresponding to a cohomology class $\beta$ in $\Cohomology^2(X,\mathbb{G_m})$.
   Given $d\in \Integers$,
   we have full subcategories $\Dqc(G,d)$ of $\Dqc(G)$ consisting of complexes with $d$-homogeneous cohomology (that is the action of 
   $\mathbb{G}_m$ on the cohomology sheaves are multiplication by $(-)^d$)\cite[Definition 4.2]{Bergh-Schurer-Decomposition-Of-Gerbes}.
    The quasi-coherent derived category 
   of $\beta$-twisted sheaves on $X$  is defined to be $\Dqc(G,1)$. It follows that the quasi-coherent derived category 
   of $\beta^d$-twisted sheaves on $X$ is equivalent to $\Dqc(G,d)$\cite[Corollary 5.8]{Bergh-Schurer-Decomposition-Of-Gerbes} and the quasi coherent derived category 
   of untwisted sheaves on $X$ are equivalent to $\Dqc(G,0)$. 
   
   \begin{Theorem}
    \label{Theorem properties of twisted sheaves}
    If 
    \[ q : G \rightarrow X \] is a $\mathbb{G}_m$-gerbe $G$ 
    corresponding to a cohomology class $\beta$ then:
    \begin{enumerate}
        \item  there is a decomposition of the derived category of sheaves on gerbes into derived categories of twisted sheaves
               on $X$ 
            \[    \prod_{d \in \Integers} \Dqc(G,d) \cong \Dqc(G)  \]
        \item if $\F_d$ and $\E_e$ are $\beta^d$ and $\beta^e$ twisted respectively then $\F_d \otimes \E_e$ and
        $\mathcal{H}om(\E_e, \F_d)$ are $\beta^{d+e}$ and $\beta^{d-e}$ twisted respectively.
        \item the map $\Left\!q^*$ is fully faithful with essential image $\Dqc(G,0)$.(as mentioned earlier untwisted sheaves
        corresponds to complexes whose cohomologies are homogeneous for the trivial action)
        \item given a map $\pi$ of algebraic stacks
          and a  2-Cartesian diagram 
           \begin{equation}
            \label{Diagram twisted sheaves functorial with respect to pullback}
            \begin{tikzcd}
            H \arrow[d, "\pi'"'] \arrow[r, "q'"] & Y \arrow[d, "\pi"]  \\
            G \arrow[r, "q"']                                                 & X                
            \end{tikzcd} 
            \end{equation}
        with $H$ a $\mathbb{G}_m$ gerbe. Then $\Left (\pi')^*$ sends $\Dqc(G,d)$ to $\Dqc(H,d)$. That is, pullback by $\pi$ sends
        $\beta^d$ twisted sheaves to $(\pi^*\beta)^d$ twisted sheaves.
    \end{enumerate}
   \end{Theorem}

   \begin{proof}
    We quote the relevant theorems in \cite{Bergh-Schurer-Decomposition-Of-Gerbes}
    \begin{enumerate}
        \item \cite[Theorem 5.4]{Bergh-Schurer-Decomposition-Of-Gerbes}
        \item \cite[Theorem 5.4]{Bergh-Schurer-Decomposition-Of-Gerbes}
        \item \cite[Proposition 5.7]{Bergh-Schurer-Decomposition-Of-Gerbes}
        \item \cite[Theorem 5.6]{Bergh-Schurer-Decomposition-Of-Gerbes}
    \end{enumerate}
   \end{proof}

\begin{Theorem}
    \label{Theorem Decomposition of Derived Category of Generalised Severi Brauer Varieties}
   For a generalized Severi-Brauer stack $\SB(k,n)_X$ over an algebraic stack $X$  there exists 
    fully faithful functors
\[         \Psi_{\alpha} : \Dqc(G(\SB(k,n)),|\alpha|) \rightarrow \Dqc(\SB(k,n))                                \]
\[        \Psi_{\alpha}(\F) = \Right\!q'_*(\Left (\pi')^* C  \otimes \Weyl^{\alpha} \Tautological)    \]  
where the notation is the same as that in \ref{Diagram Gerbe of Trivialisation}. Moreover, there is a semiorthogonal decomposition
of $\Dqc(\SB(k,n))$ into right admissible categories
    \[      \Dqc(\SB(k,n)) = \langle \{ \IM(\Psi_{\alpha}) \}_{\alpha \in \overline{B_{k,n}}}   \rangle.       \]
    In particular, there is a semiorthogonal decomposition in terms of the derived categories of twisted sheaves on $X$.
\end{Theorem}

\begin{proof}
    We will make use of the notation in (\ref{Diagram Gerbe of Trivialisation}).

    By definition of the $\mathbb{G}_m$-action on $G(\SB(k,n)_X)$, the universal bundle $V$ is in $\Dqc(G(\SB(k,n)_X),1)$ and hence $(\pi')^*(V^{\vee})$ is in
     $\Dqc(\Gr(k,V^{\vee}),-1)$
    by Theorem \ref{Theorem properties of twisted sheaves} Part(3). 
    Hence $(V^{\vee})^{\vert \alpha \vert}$ is in 
    $\Dqc(\Gr(k,\F^{\vee}),-\vert \alpha \vert)$ by Theorem  \ref{Theorem properties of twisted sheaves} Part(2) . 
    Since $\Tautological^{\vert \alpha \vert}$
     is a subsheaf of $(V^{\vee})^{\vert \alpha \vert}$, it is in $\Dqc(\Gr(k,\F^{\vee}),-\vert \alpha \vert)$ as well.
     Moreover, the surjection in equation \ref{Equation Weyl Functors} 
    \[          \Tautological^{\vert \alpha \vert} \rightarrow \Weyl^{\alpha} \Tautological.              \]
    proves $\Weyl^{\alpha} \Tautological$ is in $\Dqc(\Gr(k,\F^{\vee}),-\vert \alpha \vert)$.

    For $C$ in $\Dqc(G(\SB(k,n)_X),d)$, the complex
    $\Left (\pi')^*(C) \otimes \Weyl^{\alpha} \Tautological $ is in $\Dqc(\Gr(k,\VectorV^{\vee}),0)$ if and only if $d$ = $|\alpha|$ by
    theorem \ref{Theorem properties of twisted sheaves} part(2) and part(3). 
    We let $\Phi_\alpha(C):=\Left (\pi')^* C  \otimes \Weyl^{\alpha} \Tautological$.
    By theorem
     \ref{Theorem Derived Category of Grassmannian over algebraic stack for Dqc}, the functors $\Phi_{\alpha}$ are fully faithful and 
     have right adjoints. 
      Moreover, their restrictions 
    \[ \Phi_{\alpha} :   \Dqc(G(\SB(k,n)_X),|\alpha|) \rightarrow   \Dqc(\Gr(k,\VectorV^{\vee}),0)                                     \]   
    admit right adjoints by Theorem \ref{Theorem properties of twisted sheaves} part(1). 
    
    The semiorthogonal decomposition in 
    Theorem \ref{Theorem Derived Category of Grassmannian over algebraic stack for Dqc}
    \[    \Dqc(\Gr(k,\VectorV)) = \langle \{ \IM(\Phi_{\alpha}) \}_{\alpha \in \overline{B_{k,n}}}   \rangle                                                      \] 
    combined with \cite[Corollary 5.19]{Bergh-Schnurer-Conservative-Descent} implies the semiorthogonal decomposition
    \[ \Dqc(\Gr(k,\VectorV^{\vee}),0) =  \langle \{ \IM(\Phi_{\alpha}|_{\Dqc(G,|\alpha|)}) \}_{\alpha \in \overline{B_{k,n}}}   \rangle.  \] 
    Now $\Left (q')^*$ is a fully faithful functor onto $\Dqc(\Gr(k,\VectorV^{\vee}),0)$ by 
    theorem \ref{Theorem properties of twisted sheaves}. Hence $\Right\!q'_*$ restricted to $\Dqc(\Gr(k,\VectorV^{\vee}),0)$ is 
    fully faithful with essential image $\Dqc(\SB(k,n))$. Thus, we get a semiorthogonal decomposition 
    \[        \Dqc(\SB(k,n)) = \langle \{ \IM(\Psi_{\alpha}) \}_{\alpha \in \overline{B_{k,n}}}   \rangle                     \]
\end{proof}

\begin{Corollary}
    \label{Corollary Derived Category of Severi Brauer over algebraic stack for D perfect}
    For a generalized Severi-Brauer stack
    \[     \pi :  \SB(k,n) \rightarrow X,                                         \]
    there is a semiorthogonal decomposition of the derived categories $\D_p(\SB(k,n))$, $\D^{lb}_{pc}(\SB(k,n))$ and
    $\D_{sg}(\SB(k,n))$ into right
    admissible subcategories given by essential images of the induced fully faithful functors (for $\alpha \in \overline{B_{k,n}}$)
    \[   \Psi_{\alpha}^p : \D_p(G(\SB(k,n)),|\alpha|) \rightarrow \D_p(\SB(k,n))  \] 
    \[   \Psi_{\alpha}^{pc} :  \D^{lb}_{pc}(G(\SB(k,n)),|\alpha|) \rightarrow \D^{lb}_{pc}(\SB(k,n)) \]
    \[   \Psi_{\alpha}^{sg} :  \D_{sg}(G(\SB(k,n)),|\alpha|) \rightarrow \D_{sg}(\SB(k,n))           \]                             
    respectively.                                                        
\end{Corollary}

\begin{proof}
    This follows from Corollary \ref{Corollary Derived Category of Grassmannian over algebraic stack for D perfect} and 
    \cite[Corollary 5.19]{Bergh-Schnurer-Conservative-Descent}.
\end{proof}

\section{Applications}\label{s:flag}

We first prove that a flag variety over an algebraic stack has a  semiorthogonal decomposition. Let $\VectorV$ be a rank $n$ vector
bundle on an algebraic stack $X$. Given integers $1$ $\leq$ $i_{1}$ $<$ $i_{2}$ $<$ $\hdots$ $<$ $i_{k}$ = $n$. We denote by
 $\Fl(i_1,i_2 \hdots i_{k-1},\VectorV)_X$ the partial flag variety over $X$ parameterizing flags
\[ 0 \subset \VectorW_1 \subset \VectorW_2 \hdots \subset \VectorW_{k-1} \subset \VectorW_k = \VectorV \] 
such that $\VectorW_j$ is a vector bundle over $X$ of rank $i_j$. For each integer $j$ satisfying 1 $\leq$ $j$ $<$ $k$,
we have tautological bundles $S_j$ over $\Fl(i_1,i_2, \hdots i_{k-1},\VectorV)_X$. On the other hand, let 
$\overline{B_{i_1,i_2 \hdots i_{k-1}, n}}$ be the set of tuples $(\alpha_1, \alpha_2, \hdots, \alpha_{k-1})$ where each 
$\alpha_j$ is a partition such that the corresponding Young diagram has less than $i_j$ rows and has no more than
$i_{j+1} - i_{j}$ column. We equip it with the dictionary order.

\begin{Corollary}
	\label{Corollary Semiorthogonal Decomposition of Flags}
    Let $\Fl(i_1,i_2, \hdots i_{k-1},\VectorV)_X$ be the flag variety over an algebraic stack $X$
    \[   \pi :  \Fl(i_1,i_2, \hdots i_{k-1},\VectorV)_X  \rightarrow X                 \]
    Then there exists fully faithful Fourier Mukai functors
    \[      \Phi_{\alpha_1,\alpha_2 \hdots \alpha_{k-1}} : \Dqc(X) \rightarrow \Fl(i_1,i_2, \hdots i_{k-1},\VectorV)_X                                                        \]
    \[      \Phi_{\alpha_1,\alpha_2 \hdots \alpha_{k-1}}(\E) =  \Left\!\pi^*(\E) \otimes \Weyl^{\alpha_1}\!\Tautological_1 \otimes \Weyl^{\alpha_2}\!\Tautological_2 \hdots \otimes \Weyl^{\alpha_{k-1}}\!\Tautological_{k-1}                                             \]
    such that there is a semiorthogonal decomposition of $\Dqc(\Fl(i_1,i_2, \hdots i_{k-1},\VectorV)_X)$ into right admissible categories
    \[      \Dqc(\Gr(k,\VectorV)_{X}) = \langle \{ \IM(\Phi_{\alpha_1,\alpha_2 \hdots \alpha_{k-1}}) \}_{(\alpha_1,\alpha_2 \hdots \alpha_{k-1}) \in \overline{B_{i_1,i_2 \hdots i_{k-1}, n}}}   \rangle                    \]
    The induced functors also give a semiorthogonal decomposition of $$\D_p(\Fl(i_1,i_2, \hdots i_{k-1},\VectorV)_X),\ 
    \D^{lb}_{pc}(\Fl(i_1,i_2, \hdots i_{k-1},\VectorV)_X)\ \text{and}\ \D_{sg}(\Fl(i_1,i_2, \hdots i_{k-1},\VectorV)_X).$$   
\end{Corollary}

\begin{proof}
	Follows from theorem \ref{Theorem Grassmannian in arbitrary characteristic} and 
    corollary \ref{Corollary Derived Category of Grassmannian over algebraic stack for D perfect} since 
    $\Fl(i_1,i_2, \hdots i_{k-1},\VectorV)_X$ is a tower of Grassmannian bundle.
\end{proof}

\begin{Corollary}
	\label{Corollary Generalised Severi Brauer Over a Field}
	For a generalized Severi-Brauer variety over a field
	\[ \pi : \SB(k,n) \to K \]
    Let  $B$ be a central simple algebra representing the class $\beta$ in $\Br(K)$
    where $\beta$ is the class of $\SB(k,n)$. Then $D^{b}(\SB(k,n))$ has a semiorthogonal decomposition
    such that each component is equivalent to the derived category $D^{b}(B^{\otimes i})$  
    of bounded complexes of finitely generated modules over $B^{\otimes i}$ where $i$ $\in$ $[n]$.
    Moreover, the number of times a component equivalent to $D^{b}(B^{\otimes i})$ occurs is equal to the
    number of partitions of $i$. 
\end{Corollary}
\begin{proof}
By corollary \ref{Corollary Derived Category of Severi Brauer over algebraic stack for D perfect}  if $B$ represents $\beta$,
finite modules on $B^{i}$ are the $\beta^{i}$ twisted coherent sheaves on $K$. 
\end{proof}

\begin{Corollary}
	\label{Corollary Non representability of Generalised Severi Brauer} The
bounded derived category of a finite product of generalized Severi-Brauer varieties over
a field  admits a full exceptional collection if and only if it
splits as a finite product of Grassmannians.
\end{Corollary}

\begin{proof}
	Corollary \ref{Corollary Derived Category of Grassmannian over algebraic stack for D perfect}
    proves that a finite product of Grassmannians admits an exceptional collection, see \cite[Proposition 4.5]{Nov} for further details. 
    
    For the reverse implication, 
    suppose a finite product of generalized Severi-Brauer varieties $X$ admits a full exceptional collection.
    Then it's bounded derived category of coherent sheaves is  representable in dimension 0, see \cite[Definition 1.13, Lemma 1.19]{AB}. 
    Now corollary \ref{Corollary Generalised Severi Brauer Over a Field} gives
    us a collection that satisfies the conditions in \cite[Proposition 4.5]{Nov}. This lemma follows from \cite[Remark 4.6]{Nov}.
\end{proof}

\end{document}